\documentclass[11pt,a4]{amsart} 

\usepackage[pagewise]{lineno}
\usepackage{mathtools}						
\usepackage{thmtools}
\usepackage{enumerate}						
\usepackage{amsthm,amsmath,amsfonts,amsthm,amssymb}
\usepackage[all,cmtip]{xy}
\usepackage{graphicx}
\usepackage{comment}
\usepackage{caption}
\usepackage{marginnote}
\usepackage{braket}

\usepackage{booktabs} 
\usepackage{array} 
\usepackage{paralist} 

\usepackage[mathscr]{euscript}						 	%
\usepackage{dsfont}
\usepackage{upgreek}

\theoremstyle{plain}
\newtheorem{theorem}{Theorem}
\newtheorem{lemma}{Lemma}
\newtheorem{proposition}{Proposition}
\newtheorem{corollary}{Corollary}
\newtheorem*{theorem*}{Theorem}
\theoremstyle{definition}

\theoremstyle{remark}

\usepackage{textcomp}

\title[Splitting in a mod $\ell$ Heisenberg extension]{The decomposition of primes in nonabelian extensions of Heisenberg type and an analogue of Euler's criterion}

\author[D. Kim]{Dohyeong Kim}
\address[D. Kim]{Department of Mathematical Sciences and Institute for Data Innovation in Science, Seoul National University, GwanAk-Ro 1, Gwanak-Gu, Seoul 08826, Korea}

\author[I. Yang]{Ingyu Yang}

\address[I. Yang]{Department of Mathematical Sciences and Research Institute of Mathematics, Seoul National University, Gwanak-ro 1, Gwankak-gu, Seoul, South Korea 08826}

\date{Nov 2025} 

\newcommand{\ccc}{\mathcal{C}}
\newcommand{\cdd}{\mathcal{D}}

\newcommand{\chh}{\mathcal{H}}

\newcommand{\ckk}{\mathcal{K}}

\newcommand{\cnn}{\mathcal{N}}
\newcommand{\coo}{\mathcal{O}}

\newcommand{\crr}{\mathcal{R}}

\newcommand{\Z}{\mathbb{Z}}

\newcommand{\bff}{\mathbb{F}}

\newcommand{\bpp}{\mathbb{P}}
\newcommand{\bqq}{\mathbb{Q}}

\newcommand{\fp}{\mathfrak{p}}
\newcommand{\fq}{\mathfrak{q}}

\renewcommand{\ker}{\operatorname{Ker}}

\newcommand{\gal}{\operatorname{Gal}}
\newcommand{\frob}{\operatorname{Frob}}
\newcommand{\legendre}[2]{\genfrac{(}{)}{}{}{#1}{#2}}

\begin{document}
\begin{abstract}
 For primes $p$ and $\ell$ such that $\ell$ divides $p-1$, Hirano and Morishita constructed a nonabelian Galois extension of the function field $\bff_p(t)$ whose degree is $\ell^3$ and Galois group is of Heisenberg type. Here we analyze how primes of degree one decompose in such extensions. It amounts to investigating the decomposition of the principal ideal $(t-a)$ for $a \in \bff_p-\{0,1\}$ and our main result determines when it decomposes completely in terms of an explicit polynomial in $a$. It is reminiscent of Euler's criterion. The proof relies on both the group structure of the mod-$\ell$ Heisenberg group and the arithmetic of field extensions.
 

 
 
\end{abstract}
\maketitle

\section{Introduction}

 Since class field theory was founded in the early 20th century~\cite{CO}, the prime decomposition in abelian extensions of global field has been completely described. In particular, for quadratic extensions, the decomposition pattern is encoded in the Legendre symbol and the desired law corresponds to that of quadratic reciprocity. Combined with Euler's criterion, the law says
 \begin{align}\label{eq:1}
 \left(\frac p q \right) \equiv q^{(p-1)/2} \hspace{15pt} \text{(mod $p$)}
 \end{align}
 for distinct odd primes $p$ and $q$ such that $4$ divides $p-1$. In the language of class field theory, it relates the Artin symbols at $q$ in two abelian extensions $\bqq(\sqrt{p})$ and $\bqq(\zeta_{p})$ of $\bqq$, where $\zeta_m$ denotes a primitive $m$-th root of unity for $m\ge 1$. Despite the success of class field theory in the abelian setting, it remains a fascinating problem to find analogous laws governing the decomposition of primes in nonabelian extensions of global fields.

In this paper, we establish an analogue of \eqref{eq:1} for certain nonabelian extensions of function fields. The particular extensions we consider are those arising naturally from attempts to generalize Legendre symbols using Massey products. We mention that Massey products in Galois cohomology have received much attention in their own right, including some recent works \cite{Harpaz-Wittenberg, Maire-et-al}. For a specific context regarding the multilinear extension of the Legendre symbol, we refer interested readers to \cite{MO} for a survey and to closely related works \cite{AF, KU, AM3, AMN}. Our field extensions first appeared in \cite{HM}, which is recalled below. 

Let $\ell$ be a prime and $\bff_\ell$ the finite field with $\ell$ elements. Consider the Heisenberg group:
  \begin{equation*}
        \chh(\bff_\ell)\coloneq \Set{ 
        \begin{pmatrix}
        1 & * & * \\ 0 & 1 & * \\ 0 & 0 & 1
    \end{pmatrix}| *\in \bff_\ell}. 
    \end{equation*}
Temporarily, denote by $k$ an arbitrary field containing a primitive $\ell$-th root of unity $\zeta_\ell$. In \cite{HM}, the authors introduced a Galois extension $\crr^{(\ell)}$ of the function field $k(t)$ with Galois group $\chh(\bff_\ell)$, which we call a mod $\ell$ Heisenberg extension. It is defined as $\crr^{(\ell)} \coloneq k(t)(t^{1/\ell},(1-t)^{1/\ell},\epsilon_\ell(t)^{1/\ell})$ where $\epsilon_\ell(t)\coloneq \prod_{i=1}^{\ell-1}(1-\zeta_\ell^it^{1/\ell})^i$. They utilized $\crr^{(\ell)}$ to establish a relationship between two arithmetic invariants: the mod $\ell$ Milnor invariants $\mu_\ell(g; I)$ of Galois elements defined in \cite{KMT} and the mod $\ell$ Milnor invariants $\mu_\ell(J)$ of primes introduced in \cite{MO}. 
Also, they discovered analogies between $\crr^{(\ell)}$ and dilogarithmic functions. From this viewpoint, the authors of \cite{NA, SH} have explored related structures involving $\ell$-adic Galois polylogarithms and triple $\ell$-th power symbols. We remark that these triple power symbols are regarded as multilinear generalizations of Legendre symbols.

   



For our purpose, we will take $k=\bff_p$ for a prime $p$. Since we need a primitive $\ell$-th root of unity, we will always assume $\ell \mid p-1$. Our main result concerns the splitting of the prime ideal $(t-a)$ for $a \in \bff_p-\{0,1\}$ in the mod $\ell$ Heisenberg extension $\crr^{(\ell)}/\bff_p(t)$. To state it, we introduce certain polynomials that will play the role of $x^{(p-1)/2}$ for $\bqq(\sqrt{p})$ in \eqref{eq:1}. Put $\epsilon_{\ell,n}(t)=\prod_{i=1}^{\ell-1}(1-\zeta_\ell^{i+n}t^{1/\ell})^i$. Note that $\epsilon_{\ell,0}(t)$ is literally $\epsilon_\ell(t)$ which is a generator of $\crr^{(\ell)}$. Define the polynomial 
\begin{align}\label{def:a-l}
    A_\ell(x):=\frac{1}{\ell} \left(\sum_{j=0}^{\ell-1} \epsilon_{\ell,j}(x)^{\frac{p-1}{\ell}}\right).
\end{align}
For $\ell=2$, which is the case to be treated separately, it is simply $$A_2(x)=\frac{1}{2} \left((1-\sqrt{x})^{\frac{p-1}{2}}+(1+\sqrt{x})^{\frac{p-1}{2}}\right).$$ Although $A_2(x)$ involves $\sqrt x$, it is indeed a polynomial in $x$ since it is invariant under the substitution $\sqrt{x}\mapsto -\sqrt{x}$. Similarly, one can show that $A_\ell(x)$ is also a polynomial in $x$ using the substitution $t^{1/\ell}\mapsto\zeta_\ell t^{1/\ell}$. 

 Now, we state our main results which provide criteria for splitting of $(t-a)$ in $\crr^{(\ell)}/\bff_p(t)$. For $\ell\geq3$, we have:

\begin{restatable}{theorem}{splitodd}\label{thm:l-is-odd}
    Let $\ell$ be an odd prime, and $p$ a prime such that $\ell$ divides $p-1$. Suppose $a\in \bff_p-\{0,1\}$ satisfies $\legendre{a}{p}_\ell= \legendre{1-a}{p}_\ell=1$. Then, $(t-a)$ is totally decomposed in $\crr^{(\ell)}$ if and only if $A_\ell(a)=1$. 
\end{restatable}

In the case $\ell=2$, we have:
\begin{restatable}{theorem}{spliteven}\label{thm:l-is-two}
     Let $p$ be an odd prime and $a\in \bff_p$ such that $a\neq 0,1,\frac{1}{2}$. Let $n_a$ be the number of primes above $(t-a)$ in $\crr^{(2)}/\bff_p(t)$. Then, $n_a$ is determined by the value of $A_2(a)$;
    $$\begin{cases}
        n_a=8 &\text{if } A_2(a)=1\\
        n_a=4   & \text{if } A_2(a)=-1,0 \text{ or }A_2(a)^2=\frac{1}{1-a}\\ 
        n_a=2 & \text{if }A_2(a)^2=\frac{a}{1-a}.
    \end{cases}$$
\end{restatable}

The reason we separate the case $\ell=2$ is reflected in the group structure of $\chh(\bff_\ell)$ which has an element of order $\ell^2$ precisely when $\ell=2$.
As another feature of the case $\ell=2$, we identify the morphism between algebraic curves underlying $\crr^{(\ell)}/k(t)$ which guide the necessary computation.
For a general $\ell$, we do not have such geometric interpretation and rely on algebraic manipulations.
Despite of the difference, a key ingredient common to all cases is a symmetry with respect to action of the Galois group.
The construction of such a basis seems new for $\ell \ge 3$ and it is the most technical part of the paper.


We outline the rest of the paper. Section~\ref{sec:prelim} contains preliminaries on algebraic number theory, class field theory, and basics on mod $\ell$ Heisenberg extensions. Section~\ref{sec:l-is-two} proves Theorem~\ref{thm:l-is-two} which gives a criterion for splitting of $(t-a)$ in $\crr^{(2)}$. In Section~\ref{sec:l-is-odd}, we prove Theorem~\ref{thm:l-is-odd} which gives splitting of $(t-a)$ in $\crr^{(\ell)}$ for $\ell\geq 3$ by calculating the ring of integers explicitly. 


\subsection*{Acknowledgement}
D. Kim was supported by the National Research Foundation of Korea (NRF) grants.\footnote{No. RS-2023-00301976 funded by the Ministry of Education and No. RS-2025-02293115 funded by the Ministry of Science and ICT} I. Yang was supported by the 2025 Undergraduate Research Internship Program from College of Natural Sciences, Seoul National University.

\section{Preliminaries}\label{sec:prelim}
\subsection{Frobenius conjugacy classes}
In a global field $K$, let $\coo_K$ denote the maximal order of $K$.
For a Galois extension $K'/K$ between global fields and a prime $\fp$ of $K$, denote by $\frob_{\fp,K'} \subset\gal(K'/K)$ the Frobenius conjugacy class.
We summarize its well-known properties.
\begin{proposition}\label{prop:cft}
    Let $L$ be a Galois extension of $K$. Suppose $\fp$ is unramified in $L$, and $\coo_K/\fp$ is a finite field. 
    \begin{enumerate}
        \item $\frob_{\fp,L}=\{1\}$ if and only if $\fp$ is totally decomposed in $L$. 
        \item Let $\fq$ be a prime ideal of $\mathcal O_L$ lying over $\fp$, and $f$ be the order of $\frob_{\fp,\fq}\in \gal(L/K)$. Then, the number of prime ideals of $\coo_L$ lying above $\fp$ equals to $\frac{1}{f}[L:K]$. 
        \item If $L$ is an abelian extension of $K$, then $\frob_{\fp,L}$ is the unique element of $\gal(L/K)$ such that the map $\coo_L/\fp\coo_L\to \coo_L/\fp\coo_L$ induced by it coincides with $x\mapsto x^{\#(\coo_K/\fp)}$.
        \item If $L'$ is a Galois extension of $K$ such that $K\subset L'\subset L$, then the image of $\frob_{\fp,L} \subset \gal(L/K)$ in $\gal(L'/K)$ coincides with $\frob_{\fp,L'}$.  
    \end{enumerate}
\end{proposition}
\begin{proof}
    See \cite[Prop.\,6.29]{KA}.
\end{proof}


\subsection{auxiliary facts about integral bases}

\begin{proposition}\label{prop:compositum}
    Let $L/K$ and $L'/K$ be Galois extensions of degrees $n$ and $n'$, respectively, such that $L\cap L'=K$. Let $\omega_1,\dots,\omega_n$, and $\omega_1',\dots,\omega_{n'}'$ be integral bases of $L/K$ and $L'/K$ with discriminants $d$ and $d'$, respectively. Suppose that $d$ and $d'$ are relatively prime in the sense that $xd+x'd'=1$, for suitable $x,x'\in \coo_K$. Then, $\{\omega_i\omega_j'\colon 1 \le i \le n, 1 \le j' \le n'\}$ is an integral basis of $LL'$, of discriminant $d^{n'}d'^n$. 
\end{proposition}
\begin{proof}
    See \cite[Prop.\,2.11]{NE}.
\end{proof}

\begin{proposition}\label{prop:intclosure}
    Let $m \ge 2$ be an integer. The integral closure of $k[t]$ in $k(t^{1/m})$ is $k[t^{1/m}]$. 
\end{proposition}
\begin{proof}
Let $R$ be the integral closure of $k[t]$ in $k(t^{1/m})$. Clearly, $t^{1/m}$ is integral over $k[t]$, so $k[t^{1/m}]\subset R$. Since $k[t^{1/m}]$ is a unique factorization domain, it is integrally closed in its field of fractions. We conclude $k[t^{1/m}]=R$.
\end{proof}

\subsection{Mod $\ell$ Heisenberg Extension}\label{sec:Heisenberg}
From now on, let $p$ and $\ell$ be primes such that $\ell$ divides $p-1$. 
Here, we describe mod $\ell$ Heisenberg extensions. 

Let $\ckk^{(\ell)}$ be the extension of $k(t)$ defined by $$\ckk^{(\ell)}\coloneq k(t)(t^{1/\ell},(1-t)^{1/\ell}).$$ This is the compositum of two Kummer extension of $k(t)$ and its Galois group is $\Z/\ell\Z\times \Z/\ell\Z$. We choose its generators $\alpha,\beta$ 
which are characterized by

\begin{align*}
    \alpha(t^{1/\ell}) &\coloneq \zeta_\ell t^{1/\ell}, ~ &\alpha((1-t)^{1/\ell})&\coloneq(1-t)^{1/\ell},\\
    \beta(t^{1/\ell}) &\coloneq t^{1/\ell}, ~ &\beta((1-t)^{1/\ell})&\coloneq\zeta_\ell(1-t)^{1/\ell}.
\end{align*}
By Kummer theory, the extension $\ckk^{(\ell)}/k(t)$ is unramified outside $\{\infty,0,1\}$. It is straightforward to check that $\ckk^{(\ell)}$ corresponds to the Fermat plane curve $\ccc^{(\ell)}:X^\ell+Y^\ell=Z^\ell$ in $\bpp^2$ via the covering map $\ccc^{(\ell)}\to \bpp^1$ given by $(X:Y:Z)\mapsto(X^\ell:Z^\ell)$. 

Recall from the introduction that we set $\epsilon_\ell(t) = \prod_{i=1}^{\ell-1}(1-\zeta_\ell^it^{1/\ell})^i$ and the definition the field $\crr^{(\ell)}$ can be rephrased as $\crr^{(\ell)}=\ckk^{(\ell)}(\epsilon_\ell(t)^{1/\ell})$.

\begin{proposition}\label{prop:Heisenberg}
    Using the notation above, $\crr^{(\ell)}/k(t)$ is a Galois extension such that $\gal(\crr^{(\ell)}/k(t))$ is isomorphic to $\chh(\bff_\ell)$.
    Moreover, it is unramified outside $\{\infty, 0, 1\}$. 
\end{proposition}

\begin{proof}
    See \cite[Thm.\,2.1.4]{HM}.
\end{proof}

Note that we have extensions $\tilde{\alpha},\tilde{\beta}\in \gal(\crr^{(\ell)}/k(t))$ of $\alpha,\beta\in \gal(\ckk^{(\ell)}/k(t))$, respectively, characterized by
\begin{align*}
    \tilde\alpha(t^{1/\ell}) &\coloneq \zeta_\ell t^{1/\ell}, ~ &\tilde\alpha((1-t)^{1/\ell})&\coloneq(1-t)^{1/\ell},~  &\tilde\alpha(\epsilon_\ell(t)^{1/\ell})&\coloneq\frac{1-t^{1/\ell}}{(1-t)^{1/\ell}}\epsilon_\ell(t)^{1/\ell}\\
    \tilde\beta(t^{1/\ell}) &\coloneq t^{1/\ell}, ~ &\tilde\beta((1-t)^{1/\ell})&\coloneq\zeta_\ell(1-t)^{1/\ell} ~ &\tilde \beta(\epsilon_\ell(t)^{1/\ell})&\coloneq \epsilon_\ell(t)^{1/\ell}.
\end{align*}
Then, the correspondence
$$\tilde\alpha\mapsto \begin{pmatrix}
    1&1&0 \\ 0&1&0 \\ 0&0&1
\end{pmatrix},~~ \tilde \beta\mapsto\begin{pmatrix}
    1&0&0 \\ 0&1&1 \\ 0&0&1
\end{pmatrix},$$
induces the isomorphism $\gal(\crr^{(\ell)}/k(t))\cong \chh(\bff_\ell).$

\section{$\ell=2$}\label{sec:l-is-two}
In this section, we treat the case $\ell = 2$ by proving Theorem~\ref{thm:l-is-two}. In particular, $a$ will denote an element of $\bff_p-\{0,1\}.$ We begin with a group-theoretic lemma.

\begin{lemma}\label{lem:l-is-two-order}
    All elements in the mod $2$ Heisenberg group have order of at most $4$.
\end{lemma}
\begin{proof}
    An arbitrary $A\in \chh(\bff_2)$ can be written as $A=I+B$ where $I$ is the $3\times 3$ identity matrix and $B$ is a strict upper-triangular matrix. Then, $A^4=(I+B)^4=I+B^4=I$ gives the required result. 
\end{proof}
\begin{proposition}\label{prop:l-is-two-inert}
    $(t-a)$ is not inert in the extension $\crr^{(\ell)}/k(t)$.
\end{proposition}
\begin{proof}
    Since the Frobenius conjugacy class that corresponds to $(t-a)$ has order of at most $4$ by Lemma~\ref{lem:l-is-two-order}, Proposition~\ref{prop:cft} implies that the number of prime ideals above $(t-a)$ is at least $2$.
\end{proof}

In view of the above proposition, $(t-a)$ splits into eight, four, or two primes. As we will show below, all of three cases occur. Now, we first consider the case where $\legendre{a}{p}=\legendre{1-a}{p}=1$.

\begin{proposition}\label{prop:l-is-two-total}
    If $\legendre{a}{p}=\legendre{1-a}{p}=1$, $(t-a)$ is totally decomposed in $\crr^{(2)}$ if and only if $A_2(a)=1$.
\end{proposition}

\begin{proof}
    In this case, we use the non-singular projective model $\cdd^{(2)}$ of $\crr^{(2)}$. This is given by the plane curve $$U^2+V^2=2W^2$$ in $\bpp^2$, and the covering map $\varphi:\cdd^{(2)}\to \ccc^{(2)}$ is given by $$(U:V:W)\mapsto (W^2-U^2:UV:W^2).$$ Here, recall that $\ccc^{(2)}$ is the Fermat plane curve $X^2+Y^2=Z^2$ in $\bpp^2$ that corresponds to $\ckk^{(2)}$. Since $\legendre{a}{p}=\legendre{1-a}{p}=1$, the splitting of $(t-a)$ in $\ckk^{(2)}$ can be represented by four closed points $(\pm x:\pm y:1)$ on $\ccc^{(2)}$ where $x,y\in \bff_p$ that satisfies $x^2=a, y^2=1-a$. 

Now, $(t-a)$ is totally decomposed in $\crr^{(2)}$ when the fiber over $(\pm x:\pm y:1)$ in $\varphi:\cdd^{(2)}\to \ccc^{(2)}$ is given by two closed points over $\bff_p$. It happens if there exists two $u,v\in \bff_p$ such that $$1-u^2=x, uv=y.$$ This holds if and only if $\legendre{1-x}{p}=1$. Note that $\legendre{1-x}{p}=\legendre{1+x}{p}$ since $\legendre{1-x}{p}\legendre{1+x}{p} = \legendre{1-x^2}{p}=\legendre{1-a}{p}=1$. Thus, we have $$\legendre{1-x}{p}=\frac{1}{2}\left(\legendre{1-x}{p}+\legendre{1+x}{p} \right)=\frac{1}{2}\left((1-x)^{\frac{p-1}{2}}+(1+x)^{\frac{p-1}{2}}\right).$$
     
Recall from \eqref{def:a-l} that $A_2(a)= \frac{1}{2} \left((1-\sqrt{a})^{\frac{p-1}{2}}+(1+\sqrt{a})^{\frac{p-1}{2}}\right)$. since $x^2=a$ in $\bff_p$ and $\frac{1}{2}\left((1-x)^{\frac{p-1}{2}}+(1+x)^{\frac{p-1}{2}}\right)$ only has even degree terms, $A_2(a)=\legendre{1-x}{p}$. Thus, if $\legendre{1-x}{p}=1$, the value of the above polynomial is $1$, and if $\legendre{1-x}{p}=-1$, the value of the above polynomial is $-1$ and we obtain our desired result.
\end{proof}

Now, we show the value of $A_2(a)$ determines other quadratic residue symbols.

\begin{proposition}\label{prop:l-is-two-A_2(a)}
Assume $a \not \in\{ 0, 1, \frac 1 2\}$. The values $\legendre{a}{p}, \legendre{1-a}{p}$ are determined by the value of $A_2(a)$. More precisely, 
$$\begin{cases}
    \legendre{a}{p}= \legendre{1-a}{p}=1 & \text{if }A_2(a)=1~\text{or}~ -1\\
    \legendre{a}{p}=1,~\legendre{1-a}{p}=-1 & \text{if }A_2(a)=0 \\
    \legendre{a}{p}=-1,~ \legendre{1-a}{p}=1 & \text{if }A_2(a)^2=\frac{1}{1-a}\\
    \legendre{a}{p}= \legendre{1-a}{p}=-1 & \text{if }A_2(a)^2=\frac{a}{1-a}

\end{cases}$$

\end{proposition}
\begin{proof}
    Let $$x_n=\frac{1}{2}\left((1-\sqrt{a})^n+(1+\sqrt{a})^n\right).$$ Note that $A_2(a)=x_{\frac{p-1}{2}}$. Then, $x_n$ satisfies the recursive relation $x_{n+1}=2x_n+(a-1)x_{n-1}$ with initial value $x_0=1, x_1=1$. Also, from the definition it is straightforward to see that $x_{2n}=2x_n^2-(1-a)^n$ holds. Now, for $a\in \bff_p$, $$x_p=\frac{1}{2}\left((1-\sqrt{a})^p+(1+\sqrt{a})^p\right)=\frac{1}{2}\sum_{i=0}^p (1+(-1)^i)\binom{p}{i} (\sqrt{a})^i =1$$ where the first equality is the definition, the second equality is the binomial expansion, and the third equality holds because all but the first term vanishes. In the last equality, we rely on $p$ being odd when $\ell=2$. A similar argument shows that $$x_{p+1}=\frac{1}{2}((1-\sqrt{a})^{p+1}+(1+\sqrt{a})^{p+1})=1+a^{\frac{p+1}{2}}=1+a^{\frac{p-1}{2}}\cdot a=1+\legendre{a}{p}a .$$ 
    
    Now, consider the case $\legendre{a}{p}=1$. In this case, $x_{p+1}=1+a$ and we get $x_{p-1}=1$ from the recursive formula for $x_n$. Then, $$x_{\frac{p-1}{2}}^2=\frac{1}{2}\left(1+(1-a)^{\frac{p-1}{2}}\right)=\frac{1}{2}\left(1+\legendre{1-a}{p}\right).$$ Hence, if $\legendre{a}{p}=\legendre{1-a}{p}=1$, $x_{\frac{p-1}{2}}=A_2(a)$ is $1$ or $-1$ which is determined by $\legendre{1-\sqrt{a}}{p}$ as we argued above. If $\legendre{a}{p}=1 $ and $\legendre{1-a}{p}=-1$, then $x_{\frac{p-1}{2}}^2=0$, which is equivalent to $A_2(a)=0$. 

In the other case $\legendre{1-a}{p}=-1$, we have $x_{p+1}=1-a$. Again from the recursive formula, we get $x_{p-1}=\frac{a+1}{1-a}$. Then, $$x_{\frac{p-1}{2}}^2=\frac{1}{2}\left(\frac{a+1}{1-a}+(1-a)^{\frac{p-1}{2}}\right)=\frac{1}{2}\left( \frac{1+a}{1-a}+\legendre{1-a}{p} \right).$$ If $\legendre{1-a}{p}=1$, $x_{\frac{p-1}{2}}^2=\frac{1}{1-a}$ which is never $0,1$. If $\legendre{1-a}{p}=-1$, $x_{\frac{p-1}{2}}^2=\frac{a}{1-a}$ which is not $0,1,\frac{1}{1-a}$ except the case $a=\frac{1}{2}$. Combining these results, we proved the desired result. 
\end{proof}

Finally, we show our main result in the case $\ell=2$, Theorem~\ref{thm:l-is-two}.



\spliteven*
\begin{proof}
    We first calculate the conjugacy classes of $\chh(\bff_2)$. The identity matrix forms a conjugacy class of size one. Next, the unique non-trivial element in the center forms a conjugacy class of size one; $$\begin{pmatrix} 1&0&1\\ 0&1&0\\0&0&1\end{pmatrix}.$$ Next, non-central order two elements form two conjugacy classes; $$\Set{
    \begin{pmatrix} 1&0&0\\ 0&1&1\\0&0&1\end{pmatrix},
    \begin{pmatrix} 1&0&1\\ 0&1&1\\0&0&1\end{pmatrix}},
    \Set{\begin{pmatrix} 1&1&0\\ 0&1&0\\0&0&1\end{pmatrix},
    \begin{pmatrix} 1&1&1\\ 0&1&0\\0&0&1\end{pmatrix}}.$$
    Finally, the order 4 elements form a single conjugacy class
    $$\Set{
    \begin{pmatrix} 1&1&1\\ 0&1&1\\0&0&1\end{pmatrix},
    \begin{pmatrix} 1&1&0\\ 0&1&1\\0&0&1\end{pmatrix}}.$$
    
    From this calculation, we observe that for a representative $(a_{ij})\in \chh(\bff_2)$ of a conjugacy class, if $a_{12}=0,a_{23}=1$ or $a_{12}=1,a_{23}=0$ holds, $(a_{ij})$ has order 2. Otherwise, if $a_{12}=a_{23}=1$, $(a_{ij})$ has order 4. 

    Recall from Proposition~\ref{prop:Heisenberg} that, under the identification $$\gal(\crr^{(2)}/k(t)) \cong \chh(\bff_2),$$ the $(1,2)$- and $(2,3)$-entries are determined by the images of $\sqrt{t}$ and $\sqrt{1-t}$, respectively. More precisely, for $\sigma \in \gal(\crr^{(2)}/k(t))$, the $(1,2)$-entry of $\sigma$ is given by $$\frac{1}{2}\left(1-\frac{\sigma(\sqrt{t})}{\sqrt{t}}\right),$$ while the $(2,3)$-entry is $$\frac{1}{2}\left(1-\frac{\sigma(\sqrt{1-t})}{\sqrt{1-t}}\right).$$


    Now, we show that each $(1,2), (2,3)$-entry of the Frobenius conjugacy class corresponding to $(t-a)$ is determined by $\legendre{a}{p}$ and $\legendre{1-a}{p}$, respectively. First, restriction of $\frob_{(t-a),\crr^{(2)}}$ to $K_1=\bff_p(t)(\sqrt{t})$ is $\frob_{(t-a),K_1}$. Since $K_1/\bff_p(t)$ is abelian, its induced morphism to $\coo_{K_1}/(t-a)\to \coo_{K_1}/(t-a)$ is $x\mapsto x^p$. Hence, $\sqrt{t}$ maps to $ t^{\frac{p}{2}}=a^{\frac{p-1}{2}}\sqrt{t}$ which determines the value of $\frob_{(t-a),\crr^{(2)}}(\sqrt{t})$ as $a^{\frac{p-1}{2}}\sqrt{t}=\legendre{a}{p}\sqrt{t}$. Thus, the $(1,2)$-entry of $\frob_{(t-a),\crr^{(2)}}$ is $\frac{1}{2}(1-\legendre{a}{p})$. Similarly, the $(2,3)$-entry of $\frob_{(t-a),\crr^{(2)}}$ is $\frac{1}{2}(1-\legendre{1-a}{p})$. Combining these results with Proposition~\ref{prop:l-is-two-total} and Proposition~\ref{prop:l-is-two-A_2(a)}, we obtain the desired result.
\end{proof}

\section{$\ell\geq3$}\label{sec:l-is-odd}
In this section, we treat the case $\ell \ge 3$ by proving Theorem~\ref{thm:l-is-odd}.
In particular, $a$ will denote an element of $\bff_p-\{0,1\}$.
Put $\legendre{a}{p}_\ell=a^{\frac{p-1}{\ell}}$ in $\bff_p$.

\begin{lemma}\label{lem:l-is-odd-order}
    For $\ell\geq 3$, all elements of the mod $\ell$ Heisenberg group have order $\ell$ except the identity element.
\end{lemma}
\begin{proof}
    An arbitrary $A\in \chh(\bff_\ell)$ can be written as $A=I+B$ where $I$ is the $3\times 3$ identity matrix and $B$ is a strict upper-triangular matrix. Then, $A^\ell=(I+B)^\ell=I+B^\ell=I$ since $\ell\geq 3$.
\end{proof}

\begin{proposition}\label{prop:l-is-odd-number}
    The number of prime ideals of $\coo_{\crr^{(\ell)}}$ lying above $(t-a)$ equals $\ell^3$ or $\ell^2$.
\end{proposition}
\begin{proof}
    Since the Frobenius conjugacy class that corresponds to $(t-a)$ has order of at most $\ell$ by Lemma~\ref{lem:l-is-odd-order}, Proposition~\ref{prop:cft} implies that the number of prime ideals above $(t-a)$ is at least $\ell^2$.
\end{proof}

 Using the proposition above, it is relatively simple to determine the splitting of $(t-a)$ in $\crr^{(\ell)}/k(t)$ in certain cases.

\begin{proposition}
    If $\legendre{a}{p}_\ell\neq 1$ or $\legendre{1-a}{p}_\ell\neq 1$, then $(t-a)$ splits into $\ell^2$ primes in $\crr^{(\ell)}$.
\end{proposition}

\begin{proof}
    First, suppose $\legendre{a}{p}_\ell\neq 1$ holds. From Lemma~\ref{lem:l-is-odd-order} and Proposition~\ref{prop:l-is-odd-number}, it is enough to show that the Frobenius conjugacy class that corresponds to $(t-a)$ is not the identity matrix.
    
    Now suppose that the conjugacy class of Frobenius element $\frob_{(t-a),\crr^{(\ell)}}$ is the identity element. Denote the restriction of $\frob_{(t-a),\crr^{(\ell)}}$ to $L=k(t)(t^{1/\ell})$ by $\sigma$. Then, $\sigma$ is the identity element of $\gal(L/k(t))$ from our assumption. This contradicts with $\legendre{a}{p}_\ell=1$. Indeed, its induced map on $\coo_L/(t-a)\to \coo_{L}/(t-a)$ should be the identity map. On the other hand, $\coo_L\cong k[t^{1/\ell}]$ implies $\coo_L/(t-a)\cong k[t]/(t^\ell-a)$. From the fact that $L/k(t)$ is an abelian extenion, $\frob_{(t-a),L}(t^{1/\ell})=t^{p/\ell}=a^{p-1/\ell}t^{1/\ell}=\legendre{a}{p}_\ell t^{1/\ell}$. Thus, if $\legendre{a}{p}_\ell\neq 1$, $\frob_{(t-a),\crr^{(\ell)}}$ cannot be the identity map. This completes the argument for the first case. 

    The other case $\legendre{1-a}{p}_\ell\neq 1$ can be treated by the similar argument by using $L'=k(t)((1-t)^{1/\ell})$ instead of $L$.
\end{proof}

  So, the remaining case is the case when $a,1-a$ is both $\ell$-th power residue in $\bff_p$, which is our main case. 




    

Now, we aim to prove the Theorem~\ref{thm:l-is-odd} which provides a criterion for determining the splitting of $(t-a)$ in $\crr^{(\ell)}/\bff_p(t)$. In order to prove this theorem, we find the ring of integers of $\crr^{(\ell)}$ in steps. This part of the proof is the most technical part. 

First, we show that the ring of integers of $\ckk^{(\ell)}$ is given by $k[t][t^{1/\ell}, (1-t)^{1/\ell}]$. Next, we calculate the relative discriminant $\Delta_{\crr^{(\ell)}/k(t)}$. Finally, we find the integral basis of $\crr^{(\ell)}/k(t)$. The basis we have found is almost Galois invariant in the sense that the line spanned by them are invariant under the Galois action. 

\begin{proposition}
    The ring of integers of $\ckk^{(\ell)}$ over $\bff_p(t)$ is $\bff_p[t][t^{1/\ell},(1-t)^{1/\ell}]$. 
\end{proposition}
\begin{proof}
    Let $L=\bff_p(t^{1/\ell})$ and $L'=\bff_p((1-t)^{1/\ell})$. We wish to use Proposition~\ref{prop:compositum} for $L, L'$. 
    
    First, we show that $L\cap L'=\bff_p(t)$. Consider both $L,L'$ as a subfield of $\ckk^{(\ell)}$. As we have defined in Section~\ref{sec:Heisenberg}, the Galois group $\gal(L/\bff_p(t))$ is generated by $\alpha$, and $\gal(L'/\bff_p(t))$ is generated by $\beta$. Since $\alpha,\beta$ generates the whole Galois group $\gal(\ckk^{(\ell)})$, we conclude that $L\cap L'=\bff_p(t)$.
    
    
    Next, by Proposition~\ref{prop:intclosure}, the integral bases of $L/\bff_p(t)$, $L'/\bff_p(t)$ are given by $\Set{t^{i/\ell}}_{0\leq i<\ell}$ and $\Set{(1-t)^{j/\ell}}_{0\leq j<\ell}$, respectively. Also, the discriminants of $L/\bff_p(t)$, $L'/\bff_p(t)$ is $Ct^{\ell-1}$ and $C(1-t)^{\ell-1}$ where $C=(\prod_{i\neq j}(\zeta_\ell^i-\zeta_\ell^j))\in \bff_p^\times$. Since they are relatively prime, we are ready to apply Proposition~\ref{prop:compositum}, which gives the desired result. 
\end{proof}

Next, we calculate the discriminant of $\crr^{(\ell)}$. We begin with a lemma. 

\begin{lemma}
    In $\bff_p[t][t^{1/\ell},(1-t)^{1/\ell}]$, the ideal $(t^{1/\ell}-\zeta_\ell^n)$ is equal to the $\ell$-th power of the prime ideal $(t^{1/\ell}-\zeta_\ell^n,(1-t)^{1/\ell})$.
\end{lemma}
\begin{proof}
    To simplify the notation, put $\zeta=\zeta_\ell$ and $I_n=(t^{1/\ell}-\zeta^n,(1-t)^{1/\ell})$. First, we observe that $I_n$ is a prime ideal in $\coo_{\ckk^{(\ell)}}$. Indeed, $\bff_p[t][t^{1/\ell},(1-t)^{1/\ell}]/I_n\cong \bff_p$. 
    
    Next, we first show that $(t^{1/\ell}-\zeta^n)\supseteq I_n^\ell$ holds. We need to show $\prod_{i=1}^\ell(a_i (t^{1/\ell}-\zeta^n)+b_i(1-t)^{1/\ell})\in (t^{1/\ell}-\zeta^n)$ for all $a_i,b_i\in\bff_p[t][t^{1/\ell},(1-t)^{1/\ell}]$. After expanding the product, all the terms except $\prod_{i=1}^\ell b_i (1-t)^{1/\ell}=(\prod_{i=1}^\ell b_i)(1-t)$ is already divisible by $t^{1/\ell}-\zeta^n$. Also, since $(1-t)=\prod_{i=1}^\ell(\zeta^i-t^{1/\ell})$, the product $\prod_{i=1}^\ell(a_i (t^{1/\ell}-\zeta^n)+b_i(1-t)^{1/\ell})$ is an element of the ideal $(t^{1/\ell}-\zeta^n)$.

    Finally, we show that $(t^{1/\ell}-\zeta^n)\subseteq I_n^\ell$ holds. We only need to verify $t^{1/\ell}-\zeta^n\in I_n^\ell$. We claim that there exists $a,b\in \bff_p[t][t^{1/\ell},(1-t)^{1/\ell}]$ satisfying $t^{1/\ell}-\zeta^n=a(t^{1/\ell}-\zeta^n)^\ell+b(1-t)$. Substituting $x=t^{1/\ell}$, we observe that $(x-\zeta^n)^n$ and $(1-x^\ell)$ has greatest common divisor $x-\zeta^n$ in $\bff_p[x]$. Hence, by Bezout's identity, we get $a,b\in \bff_p[x]\subset \bff_p[t][t^{1/\ell},(1-t)^{1/\ell}]$ with $t^{1/\ell}-\zeta^n=a(t^{1/\ell}-\zeta^n)^\ell+b(1-t)$. 
\end{proof}

\begin{proposition}
    ${\crr^{(\ell)}/\ckk^{(\ell)}}$ is unramified.
\end{proposition}
\begin{proof}
    Let $\epsilon_{\ell,n}(t)=\prod_{i=1}^{\ell-1}(1-\zeta_\ell^{i+n}t^{1/\ell})^i$. Since $$\epsilon_{\ell,n}(t)=\epsilon_\ell(t)\times \frac{(1-\zeta_\ell^0 t^{1/\ell})^\ell\cdots(1-\zeta_\ell^{n-1} t^{1/\ell})^\ell}{(1-t)^n},$$ $\epsilon_{\ell,n}(t)^{1/\ell}\in \crr^{(\ell)}$. Also, since its $\ell$-th power is an element of $\coo_{\ckk^{(\ell)}}$, $\epsilon_{\ell,n}(t)^{1/\ell}\in \coo_{\crr^{(\ell)}}$. Moreover, $\Set{\epsilon_{\ell,n}(t)^{i/\ell}}_{0\leq i<\ell}$ is a basis of $\crr^{(\ell)}/\ckk^{(\ell)}$ for all $n$. Also, let $\delta\in \gal(\crr^{(\ell)}/\ckk^{(\ell)})$ as $\delta(\epsilon_\ell(t)^{1/\ell})=\zeta_\ell \epsilon_\ell(t)^{1/\ell}$ which is the generator of $\gal(\crr^{(\ell)}/\ckk^{(\ell)})$. Then, $\delta(\epsilon_{\ell,n}(t)^{1/\ell})=\zeta_\ell \epsilon_{\ell,n}(t)^{1/\ell}$. 

    Now, the discriminant includes the ideal generated by the square of determinants of matrices of the form $(\zeta_\ell^i \epsilon_{\ell,n}(t)^{j/\ell})_{ij}$ for $0\leq n<\ell$. The determinants are given by $\left(\prod_{i< j}(\zeta_\ell^i-\zeta_\ell^j)\right)\epsilon_{\ell,n}(t)^{(\ell-1)/2}$. Also, note that $(\epsilon_{\ell,n}(t))=\prod_{i=1}^{\ell-1}(1-\zeta_\ell^{i+n}t^{1/\ell})^i=\prod_{i=1}^{\ell-1}(1-\zeta_\ell^{i+n}t^{1/\ell},(1-t)^{1/\ell})^{i\ell}$ as an ideal by the lemma above. Thus, greatest common divisor of these ideals for $0\leq n\leq\ell-1$ is the unit ideal, which shows that the relative discriminant is the unit ideal. 
\end{proof}

\begin{corollary}
    The discriminant of $\crr^{(\ell)}$ is $(t(1-t))^{\ell^2(\ell-1)}$.
\end{corollary}
\begin{proof}
    Denote by $\Delta_{L/K}$ the discriminant of field extension $L/K$. By Proposition~\ref{prop:compositum}, we can see that the discriminant of $\ckk^{(\ell)}$ is given by $(t(1-t))^{\ell(\ell-1)}$. Then,  $$\Delta_{\crr^{(\ell)}/\bff_p(t)}=\cnn_{\ckk^{(\ell)}/\bff_p(t)}(\Delta_{\crr^{(\ell)}/\ckk^{(\ell)}})\Delta_{\ckk^{(\ell)}/\bff_p(t)}^\ell=(t(1-t))^{\ell^2(\ell-1)}.
    $$
\end{proof}

Finally, we give an explicit integral basis of $\crr^{(\ell)}/\bff_p(t)$. We begin with a lemma.

\begin{lemma}\label{lem:matrix}
    Let $A_1,\dots,A_\ell$ be $n\times n$ matrices. Put $\ell n\times \ell n$ matrix 
    \[
    A=\begin{pmatrix}
        A_1 & A_2 & \cdots & A_\ell\\
        A_1& \zeta_\ell A_2 & \cdots & \zeta_\ell^{\ell-1}A_\ell\\
        \vdots & \vdots & \ddots & \vdots \\
        A_1 & \zeta_\ell^{\ell-1} A_2 &\cdots &\zeta_\ell^{(l-1)^2}A_\ell
    \end{pmatrix}
    \]
    Then, $\det (A)=\prod_{i=1}^\ell\det(A_i)\times \left(\prod_{i\neq j}(\zeta_\ell^i-\zeta_\ell^j)\right)^{n}$.
\end{lemma}

\begin{proof}
    Let $$B=\begin{pmatrix}
        1 & 1 & \cdots & 1\\
        1& \zeta_\ell  & \cdots & \zeta_\ell^{\ell-1}\\
        \vdots & \vdots & \ddots & \vdots \\
        1 & \zeta_\ell^{\ell-1} &\cdots &\zeta_\ell^{(\ell-1)^2}
    \end{pmatrix}.$$ 
    From the formula for determinant of the Vandermonde matrix, $\det B=\left(\prod_{i\neq j}(\zeta_\ell^i-\zeta_\ell^j)\right)$. Now, consider the series of row operations $\Set{R_i}$ that transforms $B$ into diagonal matrix $B'$. Then, the product of diagonal elements is given by $\det B$. Now, apply the same row action on $A$ considering it as a $\ell\times \ell$ matrix with $n\times n$ block matrix as its entry. Then, we get the block diagonal matrix with diagonal entries $B'_{ii}A_i$. Hence, $\det A=\prod_{i=1}^\ell \det (B_{ii}'A_i)=(\det B)^n\times\prod_{i=1}^\ell\det(A_i)$
\end{proof}

\begin{theorem}
    An integral basis of $\crr^{(\ell)}/\bff_p(t)$ is given by $\alpha_{i,j}^0=t^{i/\ell}(1-t)^{j/\ell}$ where $0\leq i,j<\ell$ and $\alpha_{i,j}^k=t^{i/\ell}\left( \prod_{n=1}^{\ell-1}(1-\zeta_\ell^{n+j}t^{1/\ell})^{\{kn\}}\right)^{1/\ell}$ where $0\leq i,j<\ell$ and $0<k<\ell$ and $\{N\}=N-\ell[\frac{N}{\ell}]$.
\end{theorem}
Note that we have defined $$\epsilon_{\ell,n}(t)=\prod_{i=1}^{\ell-1}(1-\zeta_\ell^{i+n}t^{1/\ell})^i=\epsilon_\ell(t)\times \frac{(1-\zeta_\ell^0 t^{1/\ell})^\ell\cdots(1-\zeta_\ell^{n-1} t^{1/\ell})^\ell}{(1-t)^n}.$$ and $\alpha_{0,j}^{k}=\epsilon_{\ell,j}^{k/\ell}/(\prod_{n=1}^{\ell-1}(1-\zeta_{\ell}^{n+j}t^{1/\ell})^{[\frac{kn}{\ell}]})$.
\begin{proof}
    First, we can easily check that $\alpha_{i,j}^k\in \coo_{\crr^{(\ell)}}$ since its $\ell$-th power is included in $\coo_{\ckk^{(\ell)}}$. 
    
    Denote $\sigma_{xy}^z\in \gal(\crr^{(\ell)}/\bff_p(t))$ for $t^{1/\ell}\mapsto \zeta_\ell^x t^{1/\ell}$, $(1-t)^{1/\ell}\mapsto \zeta_\ell^y(1-t)^{1/\ell}$, $\epsilon_\ell(t)^{1/\ell}\mapsto \zeta_\ell^z \epsilon_{\ell,x}(t)^{1/\ell}$. Now we only need to show that the square of determinant of $D=(\sigma_{xy}^z(\alpha_{i,j}^{k}))$ in a certain ordering equals to $(t(1-t))^{\ell^2(\ell-1)}$ up to some unit $C\in \bff_p^\times$.

    Now, note that $\sigma_{xy}^z(\alpha_{ij}^0)=\zeta_\ell^{ix+jy}\alpha_{i,j}^0$ and $\sigma_{xy}^z(\alpha_{i,j}^k)=\zeta_\ell^{ix-jy+kz}\alpha_{i,j+x}^k$ for $k>0$. 

    Varying $k,z$, $D$ can be represented as 
\[
\bordermatrix{
      & k=0 & k=1 & \cdots & k=\ell-1 \cr
z=0 & A_0 & A_1 & \cdots & A_{\ell-1} \cr
z=1 & A_0 & \zeta_\ell A_1 & \cdots & \zeta_\ell^{\ell-1} A_{\ell-1} \cr
\vdots & \vdots & \vdots & \ddots & \vdots \cr
z=\ell-1 & A_0 & \zeta_\ell^{\ell-1} A_1 & \cdots & \zeta_\ell^{(\ell-1)^2} A_{\ell-1}
}.
\]
    By Lemma~\ref{lem:matrix}, $\det D$ can be expressed by product of $\det A_k$ and some unit $C\in \bff_p^{\times}$. Next, we calculate each $\det A_k$.

    First, $\det A_0$ is exactly the discriminant of $\ckk^{(\ell)}/\bff_p(t)$ which is $C\times(t(1-t))^{\frac{\ell(\ell-1)}{2}}$. Next, we calculate $\det A_k$ for $k>0$. 

    Now, varying $y,j$, we can express $A_k$ by 
\[
\bordermatrix{
      & j=0 & j=1 & \cdots & j=\ell-1 \cr
y=0 & B_{0,k} & B_{1,k} & \cdots & B_{\ell-1,k} \cr
y=1 & B_{0,k} & \zeta_\ell^{-k} B_{1,k} & \cdots & \zeta_\ell^{-k(\ell-1)} B_{\ell-1,k} \cr
\vdots & \vdots & \vdots & \ddots & \vdots \cr
y=\ell-1 & B_{0,k} & \zeta_\ell^{-k(\ell-1)} B_{1,k} & \cdots & \zeta_\ell^{-k(\ell-1)^2} B_{\ell-1,k}
}.
\] 
    Again using Lemma~\ref{lem:matrix}, $\det A_k$ can be expressed by product of $\det B_{j,k}$ and some unit $C\in \bff_p^{\times}$. Now, we explicitly write $B_{j,k}$ as

\[
\bordermatrix{
      & i=0 & i=1 & \cdots & i=\ell-1 \cr
x=0 & \alpha_{0,j}^k & t^{1/\ell}\alpha_{0,j}^k & \cdots & t^{\ell-1/\ell}\alpha_{0,j}^k \cr
x=1 & \alpha_{0,j+1}^k & \zeta_\ell t^{1/\ell}\alpha_{0,j+1}^k & \cdots & \zeta_\ell^{\ell-1} t^{\ell-1/\ell}\alpha_{0,j+1}^k \cr
\vdots & \vdots & \vdots & \ddots & \vdots \cr
x=\ell-1 & \alpha_{0,j+\ell-1}^k & \zeta_\ell^{\ell-1} t^{1/\ell}\alpha_{0,j+\ell-1}^k & \cdots & \zeta_\ell^{(\ell-1)^2} t^{\ell-1/\ell}\alpha_{0,j+\ell-1}^k
}.
\] 

Thus, $\det B_{j,k}$ can be given as $t^{\frac{\ell-1}{2}}\prod_{j=0}^{\ell-1} \alpha_{0,j}^k$ by Lemma~\ref{lem:matrix}. Adding up, $$\det A_k=C\cdot\prod_{j=0}^{\ell-1} t^{\frac{\ell-1}{2}}(\alpha_{0,j}^k)^{\ell}=Ct^{\frac{\ell(\ell-1)}{2}}\prod_{j=0}^{\ell-1} (\alpha_{0,j}^k)^\ell.$$ 

Thus, $$\det D=Ct^{\frac{\ell^2(\ell-1)}{2}}(t-1)^{\frac{\ell(\ell-1)}{2}}\prod_{k=1}^{\ell-1}\prod_{j=0}^{\ell-1} (\alpha_{0,j}^k)^\ell.$$

Finally, we only need to calculate $$\prod_{k=1}^{\ell-1}\prod_{j=0}^{\ell-1} (\alpha_{0,j}^k)^\ell=\prod_{j=0}^{\ell-1}\prod_{k=1}^{\ell-1} (\alpha_{0,j}^k)^\ell.$$

Recall that $\alpha_{0,j}^k=\left( \prod_{n=1}^{\ell-1}(1-\zeta_\ell^{n+j}t^{1/\ell})^{\{kn\}}\right)^{1/\ell}.$ For $n\in (\Z/\ell\Z)^\times$, the set $\Set{kn|k\in (\Z/\ell\Z)^\times}$ is exactly $(\Z/\ell\Z)^\times$. Hence, $$\prod_{k=1}^{\ell-1} \alpha_{0,j}^k=\prod_{n=1}^{\ell-1}\prod_{k=1}^{\ell-1} (1-\zeta_\ell^{n+j}t^{1/\ell})^{\{kn\}/\ell}=\prod_{n=1}^{\ell-1}(1-\zeta_\ell^{n+j}t^{1/\ell})^{\frac{\ell-1}{2}}.$$ Then, $$\prod_{j=0}^{\ell-1}\prod_{k=1}^{\ell-1} (\alpha_{0,j}^k)^\ell=\prod_{j=0}^{\ell-1}\prod_{n=1}^{\ell-1}(1-\zeta_\ell^{n+j}t^{1/\ell})^{\frac{\ell(\ell-1)}{2}}=(1-t)^{\frac{\ell(\ell-1)^2}{2}}.$$

So, we get $$(\det D)^2=C (t(t-1))^{\ell^2(\ell-1)}=\Delta_{\crr^{(\ell)}/\bff_p(t)}$$ implying that $\left\{\alpha_{i,j}^k\right\}$ is an integral basis. 
\end{proof}

Before we prove our main result, we show that $A_\ell(x)$ is indeed a polynomial. Recall that we have defined $A_\ell(x) = \frac{1}{\ell} \left(\sum_{j=0}^{\ell-1} \epsilon_{\ell,j}(x)^{\frac{p-1}{\ell}}\right)$ where $\epsilon_{\ell,n}(x)=\prod_{i=1}^{\ell-1}(1-\zeta_\ell^{i+n}x^{1/\ell})^i$. 
\begin{lemma}
    $A_\ell(x)$ is a polynomial in $x$.
\end{lemma}

\begin{proof}
    Since $\epsilon_{\ell,n}(x)$ is a polynomial in $x^{1/\ell}$, $A_\ell(x)$ is also a polynomial in $x^{1/\ell}$. Substituting $\zeta_\ell x^{1/\ell}$ to $x^{1/\ell}$, $A_\ell(x)$ stays invariant. So, $A_\ell(x)$ only has degree of multiple of $\ell$ as a polynomial in $x^{1/\ell}$.
\end{proof}

\splitodd*

\begin{proof}
    Choose $x,y\in \bff_p$ such that $x^\ell=a,y^\ell=1-a$. Then, $x^\ell=a$ implies $t-a=\prod_{i=0}^{\ell-1}(t^{1/\ell}-\zeta_\ell^i x)$, which in turn implies that the prime ideal $(t^{1/\ell}-\zeta_\ell^ix,(1-t)^{1/\ell}-\zeta_\ell^j y)$ divides $(t-a)$ for all $0\leq i,j<\ell$. These $\ell^2$ primes are all distinct. Since $\ckk^{(\ell)}$ is a field extension of degree $\ell^2$, the factorization of $(t-a)$ is $(t^{1/\ell}-\zeta_\ell^ix,(1-t)^{1/\ell}-\zeta_\ell^j y)$ where $0\leq i,j<\ell$.

    Next, we need a lemma. Note that we can define $\epsilon_\ell(a)=\prod_{k=1}^{\ell-1}(1-\zeta_\ell^{k+i}x)^k$ by using $\zeta_\ell^ix$ as the preferred value of $a^{1/\ell}$.
    \begin{lemma}\label{lem:6}
        
        $(t^{1/\ell}-\zeta_\ell^ix,(1-t)^{1/\ell}-\zeta_\ell^j y)$ is totally decomposed in $\crr^{(\ell)}/\ckk^{(\ell)}$ if and only if $\epsilon_\ell(a)$ is an $\ell$-th power residue or equivalently $\epsilon_\ell(a)^{\frac{p-1}{\ell}}=1$.
    \end{lemma}

    \begin{proof}
        To prove the lemma, we show that 
        $$\bff_p[t][\alpha_{i,j}^k]/(t^{1/\ell}-\zeta_\ell^ix,(1-t)^{1/\ell}-\zeta_\ell^j y)$$
        is isomorphic to $\bff_p[\epsilon_\ell(t)^{1/\ell}]/(\epsilon_\ell(t)-\epsilon_\ell(a))$.

        We first claim that there is an unique homomorphism $\phi:\bff_p[t][\alpha_{i,j}^k]\to \bff_p[\epsilon_\ell(t)^{1/\ell}]/(\epsilon_\ell(t)-\epsilon_\ell(a))$ by $\phi(t^{1/\ell})=\zeta_\ell^i x$ such that $\phi(t^{1/\ell})=\zeta_\ell^i x$, $\phi((1-t)^{1/\ell})=\zeta_\ell^j y$ and $\phi(\epsilon_\ell(t)^{1/\ell})=\epsilon_\ell(t)^{1/\ell}$. Recall that $$\epsilon_{\ell,n}(t)=\prod_{i=1}^{\ell-1}(1-\zeta_\ell^{i+n}t^{1/\ell})^i=\epsilon_\ell(t)\times \frac{(1-\zeta_\ell^0 t^{1/\ell})^\ell\cdots(1-\zeta_\ell^{n-1} t^{1/\ell})^\ell}{(1-t)^n}$$ and $$\alpha_{0,j}^{k} =\epsilon_{\ell,j}^{k/\ell} \prod_{n=1}^{\ell-1}(1-\zeta_{\ell}^{n+j}t^{1/\ell})^{[\frac{kn}{\ell}]}.$$ Hence, we can define $\phi$ from $\phi(t^{1/\ell})$ and $\phi((1-t)^{1/\ell})$. Also, since the minimal polynomial of $\epsilon_\ell(t)^{1/\ell}$ in $\crr^{(\ell)}/\ckk^{(\ell)}$ is $X^\ell-\epsilon_\ell(t)=0$, $\phi$ is well defined.

        Next, we show that $\phi$ induces the desired isomorphism. The homomorphism $\phi$ defined above is clearly a surjection. Hence, it is enough to show that $\ker \phi=(t^{1/\ell}-\zeta_\ell^ix,(1-t)^{1/\ell}-\zeta_\ell^j y)$. By our construction, $\ker\phi\supseteq(t^{1/\ell}-\zeta_\ell^ix,(1-t)^{1/\ell}-\zeta_\ell^j y)$. On the other hand, let $f\in \bff_p[t][\alpha_{i,j}^k]$ be an element of $\ker \phi$. Then, by the relation above, we can express $f$ as a sum of a polynomial of $\epsilon_\ell(t)^{1/\ell}$ and an element of $(t^{1/\ell}-\zeta_\ell^ix,(1-t)^{1/\ell}-\zeta_\ell^j y)$. So, we may assume that $f$ is a polynomial of $\epsilon_\ell(t)^{1/\ell}$ of degree less than $\ell$. Since $\phi(\epsilon_\ell(t)^{1/\ell})=\epsilon_\ell(t)^{1/\ell}$, we find that $f=0$. Thus, $\ker \phi\supseteq(t^{1/\ell}-\zeta_\ell^ix,(1-t)^{1/\ell}-\zeta_\ell^j y)$ holds which proves the desired isomorphism.

        From the isomorphism $$\bff_p[t][\alpha_{i,j}^k]/(t^{1/\ell}-\zeta_\ell^ix,(1-t)^{1/\ell}-\zeta_\ell^j y)\cong\bff_p[\epsilon_\ell(t)^{1/\ell}]/(\epsilon_\ell(t)-\epsilon_\ell(a)),$$ the number of prime ideals above $(t^{1/\ell}-\zeta_\ell^ix,(1-t)^{1/\ell}-\zeta_\ell^j y)$ in $\crr^{(\ell)}$ equals to the number of those above $(\epsilon_\ell(t)-\epsilon_\ell(a))$ in $\bff_p[\epsilon_\ell(t)^{1/\ell}]$. This proves the lemma. 
    \end{proof}

    Note that $\epsilon_\ell(a)^{\frac{p-1}{\ell}}=\epsilon_{\ell,n}(a)^{\frac{p-1}{\ell}}$ holds by the definition $\epsilon_{\ell,n}(t)=\epsilon_\ell(t)\times \frac{(1-\zeta_\ell^0 t^{1/\ell})^\ell\cdots(1-\zeta_\ell^{n-1} t^{1/\ell})^\ell}{(1-t)^n}$. Hence, we obtain $\epsilon_\ell(a)^{\frac{p-1}{\ell}}=A_\ell(a)$ which implies the desired result from Lemma~\ref{lem:6}.
\end{proof}

\bibliographystyle{plain}
\bibliography{ref}

\end{document}